\documentclass[11pt,a4paper,twoside]{amsart}

\usepackage{amsfonts}
\usepackage{amsmath}
\usepackage{amssymb}
\usepackage{latexsym}
\usepackage[all]{xy}

\newcommand{\cf}{{\em cf.}\ }

\newcommand{\ca}{{\mathcal A}}

\newcommand{\cd}{{\mathcal D}}

\newcommand{\ch}{{\mathcal H}}

\newcommand{\cp}{{\mathcal P}}

\newcommand{\ct}{{\mathcal T}}

\newcommand{\cx}{{\mathcal X}}

\newcommand{\per}{\mbox{per}\,}
\newcommand{\Mod}{\mbox{Mod}\,}
\newcommand{\Hom}{\mbox{Hom}\,}

\newcommand{\Z}{{\mathbb Z}}
\newcommand{\ten}{\otimes}
\newcommand{\iso}{\stackrel{_\sim}{\rightarrow}}

\title{The Hall algebra of a spherical object}

\author{Bernhard Keller}
\address{Bernhard Keller\\
Universit\'e Paris Diderot -- Paris 7\\
UFR de Math\'ematiques\\
Institut de Math\'ematiques de Jussieu, UMR 7586 du CNRS \\
Case 7012\\
2, place Jussieu \\
75251 Paris Cedex 05\\
France }
\email{keller@math.jussieu.fr}

\author{Dong Yang}
\address{
Dong Yang\\
Institut de Math\'ematiques de Jussieu, UMR 7586 du CNRS \\
Th\'eorie des groupes\\
Case 7012\\
2 place Jussieu\\
75251 Paris Cedex 05\\
France}
\email{yang@math.jussieu.fr}

\author{Guodong Zhou}
\address{Guodong Zhou\\
Universit\"at zu K\"oln\\
Mathematisches Institut\\
Weyertal 86--90\\
D-50931 K\"oln\\
Germany}
\email{gzhou@mi.uni-koeln.de}

\date{October 5, 2008. Last modified on \today.}

\newtheorem{Thm}{Theorem}[section]

\newtheorem{Prop}[Thm]{Proposition}
\newtheorem{Lem}[Thm]{Lemma}
\newtheorem{Cor}[Thm]{Corollary}

\newtheorem{Examples}[Thm]{Examples}
\newtheorem*{Rem}{Remark}
\newtheorem*{Rems}{Remarks}

\newcommand{\ra}{\rightarrow}

\begin{document}

\begin{abstract} We determine the Hall algebra, in the sense of
To\"en, of the algebraic triangulated category generated by a spherical object.
\end{abstract}

\maketitle

\section{Introduction}\label{S:introduction}

This note is motivated by recent developments in the
categorification of cluster algebras and cluster varieties. Let us
recall the context: To a finite quiver $Q$ without loops and without
$2$-cycles, one can associate the cluster algebra $\ca_Q$ and the
cluster variety $\cx_Q$ (endowed with a Poisson structure), \cf
\cite{FominZelevinsky02} and \cite{FockGoncharov03}.  If $Q$ does
not have oriented cycles, we have at our disposal a very good
categorical model for the combinatorics of the cluster algebra
$\ca_Q$, \cf the surveys \cite{BuanMarsh06} \cite{Reiten06}
\cite{Ringel07} \cite{Keller08c}.  In contrast, for the moment,
there is no corresponding theory for the cluster variety $\cx_Q$.
Ongoing work by Kontsevich-Soibelman \cite{KontsevichSoibelman07},
Bridgeland \cite{Bridgeland08} and others shows that there is a
close link between the quantized version \cite{FockGoncharov03} of
$\cx_Q$ and the Hall algebra \cite{Toen06} of a certain triangulated
$3$-Calabi-Yau category $\ct_Q$ associated with $Q$. The category
$\ct_Q$ can be described as the algebraic triangulated category generated by
the objects in a `generic' collection of $3$-spherical objects whose
extension spaces have dimensions encoded by the quiver $Q$.
Alternatively, it may be described as the derived category of dg
modules with finite-dimensional total homology over the Ginzburg dg
algebra \cite{Ginzburg06} associated with $Q$ and a generic
potential. In this note, we consider the case where $Q$ is reduced
to a single vertex without any arrows. This amounts to considering
the (algebraic) triangulated category $\ct_Q$ generated by a single
spherical object. We first show that this category is indeed
well-determined up to a triangle equivalence
(Theorem~\ref{T:triangulated}).  Then we classify the objects of
$\ct_Q$ (Theorem~\ref{T:classification}, due to P.~J{\o}rgensen
\cite{Joergensen04}), compute the  Hall algebra of $\ct_Q$ (Theorem
~\ref{T:structure}) and establish the link with the cluster variety,
which in this case is just a one-dimensional torus
(section~\ref{S:cluster-variety}). The Hall algebra of the algebraic
triangulated category generated by a spherical object of arbitrary
dimension can be determined similarly. We give the result in
section~\ref{S:arbitrary-dimension}.  For the classification
theorem, we establish more generally the classification of the
indecomposable objects in a triangulated category admitting a
generator whose graded endomorphism algebra is hereditary, a result
which may be useful in other contexts as well.

\section*{Acknowledgments}

The second author gratefully acknowledges a postdoctoral fellowship
gran\-ted by the Universit\'e Pierre et Marie Curie, during which
this work was carried out. The third author benefited from financial
support via postdoctoral fellowships from the Ecole Polytechnique
and from the network `Representation theory of algebras and
algebraic Lie theory'. Both of them would like to thank the first
author and Professor Steffen K\"{o}nig for their support. All three
authors thank Peter J{\o}rgensen for pointing out references
\cite{Joergensen04} and \cite{Joergensen06}.

\section{The triangulated category generated by a
spherical object}
\label{S:triangulated}

Let $k$ be a field and $\ct$ a $k$-linear
algebraic triangulated category (\cf section~3.6 of
\cite{Keller06d} for this terminology). We write
$\Sigma$ for the suspension functor of $\ct$. We assume that
$\ct$ is idempotent complete, i.e. each idempotent
endomorphism of an object of $\ct$ comes from a direct
sum decomposition.

Let $d$ be an integer and $G$ a $d$-spherical object
of $\ct$. This means that the graded endomorphism algebra
\[
B=\bigoplus_{p\in\Z} \ct(G,\Sigma^n G)
\]
is isomorphic to the cohomology of the $d$-sphere, i.e. to
$k\times k$ if $d=0$ and
to $k\langle s\rangle/(s^2)$, where $s$ is of degree $d$, if $d\neq 0$.
We also view $B$ as a dg algebra whose differential vanishes.
We refer to section~3 of \cite{Keller06d} for the definition of the derived
category $\cd(B)$ and the perfect derived category $\per(B)$.
We say that $G$ {\em classically generates} $\ct$ if $\ct$ coincides
with its smallest triangulated subcategory stable under
taking direct factors and containing $G$.

\begin{Thm} \label{T:triangulated} If $G$ classically generates $\ct$,
there is a triangle equivalence from $\ct$ to the perfect
derived category of $B$.
\end{Thm}

\begin{proof} According to theorem~7.6.0.6 of \cite{Lefevre03},
there is a triangle equivalence between $\ct$ and the perfect
derived category of a minimal strictly unital $A_\infty$-algebra
whose underlying graded algebra is $B$. This $A_\infty$-structure
is given by linear maps
\[
m_p: (ks)^{\ten p} \to B
\]
defined for $p\geq 3$ and homogeneous of degree $2-p$. For
degree reasons, these maps vanish. The claim follows because
the perfect derived category of $B$ considered as a dg algebra is equivalent
to  the perfect derived category  of $B$ considered as an $A_\infty$-algebra
by Lemma~4.1.3.8 of \cite{Lefevre03}.
\end{proof}

Thus, it makes sense to speak about `the' algebraic triangulated
category generated by a spherical object of dimension $d$. Notice
that Koszul duality provides us with another realization of this
category: If $d\neq 1$, we have a triangle equivalence
\[
\per(B) \iso \cd_{fd}(A) \;\; ,
\]
where, for $d=0$, we have $A=B$, and for $d\neq 0$, the dg algebra
$A$ is the free dg algebra on a closed generator $t$ of degree
$-d+1$. Here the category $\cd_{fd}(A)$ is the full subcategory of
the derived category $\cd A$ formed by the dg modules whose homology
is of finite total dimension. If $d$ equals $3$, then $A$ is the
Ginzburg algebra \cite{Ginzburg06} associated with the quiver $A_1$.
 If $d$ equals $1$, then $\per(B)$ is triangle equivalent to the full
 subcategory of $\mathcal{D}(k[t])$, where $t$ is of degree $0$, formed by the
  dg modules whose homology is of finite total dimension and annihilated by some power of $t$.

\section{Classification}\label{S:classification}
In this section, we present a general classification theorem for
indecomposable objects in a triangulated category admitting a
`generator' $G$ whose graded endomorphism algebra is hereditary.
We first consider the case where $G$ compactly generates a triangulated
category with arbitrary direct sums. Then we consider the case where
$G$ is a classical generator. We apply it to the perfect and the
finite dimensional derived categories of the Ginzburg algebra of
type $A_1$.

\subsection{Compactly generated case}
Let $k$ be a commutative ring, and $\mathcal{T}$ a $k$-linear triangulated
category with suspension functor $\Sigma$. Assume $\mathcal{T}$ has
arbitrary direct sums. Let $G$ be a {\em compact generator} for
$\mathcal{T}$, i.e. the functor $\mathrm{Hom}_{\mathcal{T}}(G,?)$
commutes with arbitrary direct sums, and given an object $X$ of
$\mathcal{T}$, if $\text{Hom}_{\mathcal{T}}(G, \Sigma^p X)$ vanishes
for all integers $p$, then $X$ vanishes. Let
\[A=\bigoplus_{p\in\mathbb{Z}}\mathrm{Hom}_{\mathcal{T}}(G,\Sigma^{p}G)\]
be the graded endomorphism algebra of $G$. Then for any object $X$
of $\mathcal{T}$, the graded vector space
\[\bigoplus_{p\in\mathbb{Z}}\mathrm{Hom}_{\mathcal{T}}(G,\Sigma^{p}X)\]
has a natural graded (right) module structure over $A$.  We define a
functor $$F:\mathcal{T}\rightarrow \mathrm{Grmod}(A), X\mapsto
\bigoplus_{p\in\mathbb{Z}}\mathrm{Hom}_{\mathcal{T}}(G,\Sigma^{p}X).$$
Notice that since $G$ is a compact generator, a morphism of
$\mathcal{T}$ is invertible if and only if its image under $F$ is
invertible.

We say that $A$ is {\em graded hereditary}, if the category
$\mathrm{Grmod}(A)$ of graded $A$-modules is hereditary, or in other
terms, each subobject of a projective object of $\mathrm{Grmod}(A)$
is projective.

\begin{Thm} With the notations above, suppose that $A$ is graded hereditary.
 The functor $F:\mathcal{T}\rightarrow \mathrm{Grmod}(A)$
 is full, essentially surjective, and its kernel has square zero. In particular, it
 induces a bijection from the set of isoclasses of objects (respectively, of
 indecomposable objects) of $\mathcal{T}$ to that of $\mathrm{Grmod}(A)$.
\end{Thm}

\begin{Rems}

 a)  Notice that we have an isomorphism of functors $F\circ\Sigma\simeq [1]\circ F$, where $[1]$ denotes the shift functor in $\mathrm{Grmod}(A)$.

b) The functor $F$ is obviously a homological functor. We will use this fact implicitly.

\end{Rems}

The theorem is a consequence of the following lemmas.

For  a class $\mathcal{S}$ of objects of an additive category $\mathcal{A}$
with arbitrary direct sums, we denote by $\mathrm{Add}(\mathcal{S})$ the closure
of $\mathcal{S}$ under taking all direct sums and direct summands.

\begin{Lem}\label{L:equivadd}
 a) The functor $F:\mathcal{T}\rightarrow \mathrm{Grmod}(A) $ induces an equivalence between $\mathrm{Add}(\Sigma^p G|p\in\mathbb{Z})$ and $\mathrm{Add}(A[p]|p\in\mathbb{Z})$.

b) An object $X$ belongs to $\mathrm{Add}(\Sigma^p G|p\in\mathbb{Z})$ if and only if $FX$ belongs to $\mathrm{Add}(A[p]|p\in\mathbb{Z})$.
\end{Lem}
\begin{proof} a)
 By definition, we have
 $$
 \mathrm{Hom}_{\mathcal{T}}(G,X)=(FX)_0 \cong \mathrm{Hom}_{\mathrm{Grmod}(A)}(A,FX)
 $$
 for any object $X$ in $\mathcal{T}$, and so the map
 $$
 F(G,X): \mathrm{Hom}_{\mathcal{T}}(G, X) \rightarrow \mathrm{Hom}_{\mathrm{Grmod}(A)}(FG, FX)$$
 is bijective. Therefore the map
 $$
 F(G_0,X): \mathrm{Hom}_{\mathcal{T}}(G_0, X) \rightarrow \mathrm{Hom}_{Grmod(A)}(G_0, X)
 $$
 is an isomorphism for any $G_0$ in $\mathrm{Add}(\Sigma^p G|p\in\mathbb{Z})$
 and any $X$ in $\mathcal{T}$. Taking $X$ in  $\mathrm{Add}(\Sigma^p G|p\in\mathbb{Z})$,
 this proves that, considered as a functor from  $\mathrm{Add}(\Sigma^p G|p\in\mathbb{Z})$
 to $\mathrm{Add}(A[p]|p\in\mathbb{Z})$, the functor $F$ is fully faithful.
 Moreover, since $G$ is compact, $F$ commutes with arbitrary
 coproducts. The proof of essential surjectivity is therefore easy.

 b) The necessity of the condition is also an easy consequence of the fact that
 $F$ commutes with arbitrary coproducts.
Let us prove the sufficiency. Let $X$ be an object of $\ct$ such
that there is an isomorphism $f:FG_0\rightarrow FX$ in
$\mathrm{Add}(A[p]|p\in\mathbb{Z})$ for some $G_0\in
\mathrm{Add}(\Sigma^p G|p\in\mathbb{Z})$.  We can lift $f$ to a
morphism $\tilde{f}:G_0\rightarrow X$ in $\mathcal{T}$. As we have
observed above, since $G$ is a compact generator, a morphism of
$\ct$ is invertible iff its image under $F$ is invertible. Since we
have $F\tilde{f}=f$, it follows that $\tilde{f}$ is invertible and
$X$ is isomorphic to $G_0$.
\end{proof}

It is well known that the class of projective objects of $\mathrm{Grmod}(A)$ is exactly $\mathrm{Add}(A[p]|p\in\mathbb{Z})$.

\begin{Lem}\label{L:surj}
 The functor $F$ is essentially surjective.
\end{Lem}
\begin{proof}
 Let $M$ be an object in $\mathrm{Grmod}(A)$. We want to find $X$ in $\mathcal{T}$ such that $FX\cong M$.
Since $\mathrm{Grmod}(A)$ is hereditary, there exists a short exact sequence of graded $A$-modules
$$0\rightarrow P_1\stackrel{u}{\rightarrow} P_0\rightarrow M\rightarrow 0$$
with $P_0,P_1\in \mathrm{Add}(A[p]|p\in\mathbb{Z})$.
By Lemma~\ref{L:equivadd} a), there are two objects $G_0,G_1$ in $\mathrm{Add}(\Sigma^p G|p\in\mathbb{Z})$ and $v\in\mathrm{Hom}_{\mathcal{T}}(G_1,G_0)$ such that
\[F(G_1\stackrel{v}{\rightarrow} G_0)\cong (P_1\stackrel{u}{\rightarrow} P_0).\]
Let $X$ be a cone of $v$, i.e. we have a triangle
\[G_1\stackrel{v}{\rightarrow} G_0\stackrel{w}{\rightarrow} X\rightarrow\Sigma G_1 \]
in $\mathcal{T}$. We apply the homological functor $F$ and obtain an exact sequence
\[F(\Sigma^{-1}X)\rightarrow FG_1\stackrel{Fv}{\rightarrow}FG_0\stackrel{Fw}{\rightarrow}FX\rightarrow F(\Sigma G_1)\stackrel{F(\Sigma v)}{\rightarrow} F(\Sigma G_0).\]
Recall that $Fv$ is injective, and so is $F(\Sigma v)=(Fv)[1]$. Therefore $Fw$ is surjective. Then we obtain a short exact sequence in $\mathrm{Grmod}(A)$
\[0\rightarrow P_1\stackrel{u}{\rightarrow} P_0\rightarrow FX\rightarrow 0,\]
and hence $FX\cong M$.
\end{proof}

\begin{Lem}\label{L:full}
 The functor $F$ is full.
\end{Lem}
\begin{proof}
 We prove this in three steps.

{\em Step 1:} By the first paragraph of the proof of Lemma~\ref{L:equivadd},  the map $$F(G,X): \mathrm{Hom}_{\mathcal{T}}(G, X) \rightarrow \mathrm{Hom}_{\mathrm{Grmod}(A)}(G, X)$$  is an isomorphism for any $G_0$ in $\mathrm{Add}(\Sigma^p G|p\in\mathbb{Z})$ and any $X$ in $\mathcal{T}$.

{\em Step 2:} Let $X$ be an object of $\mathcal{T}$. We will show that there exists a triangle
\[G_1\rightarrow G_0\rightarrow X\rightarrow \Sigma G_1\]
in $\mathcal{T}$ such that $G_0,G_1$ belong to $\mathrm{Add}(\Sigma^p G|p\in\mathbb{Z})$.
We choose $w:G_0\rightarrow X$ such that $Fw$ is surjective. We form the triangle
\[Y\rightarrow G_0\stackrel{w}{\rightarrow} X\rightarrow\Sigma Y.\]
We apply $F$ and obtain an exact sequence
\[
\xymatrix@C=0.4cm{F(\Sigma^{-1}G_0) \ar[rr]^-{F(\Sigma^{-1}w)} & &
F(\Sigma^{-1}X) \ar[r] &  FY \ar[r] &  FG_0 \ar[r]^{Fw} &
FX \ar[r] & F(\Sigma Y).}
\]
Both $Fw$ and $F(\Sigma^{-1}w)=(Fw)[-1]$ are surjective, so we obtain
a short exact sequence
\[
0\rightarrow FY\rightarrow FG_0\rightarrow FX\rightarrow 0.
\]
Thus $FY$ belongs to $\mathrm{Add}(A[p]|p\in\mathbb{Z})$ since $\mathrm{Grmod}(A)$ is hereditary. By Lemma~\ref{L:equivadd} b), the object $Y$ belongs to $\mathrm{Add}(\Sigma^p G|p\in\mathbb{Z})$. Now it suffices to take $G_1=Y$.

{\em Step 3:} Let $X,Y$ be objects in $\mathcal{T}$. By Step 2, there is a triangle in $\mathcal{T}$
\[G_1\rightarrow G_0\rightarrow X\rightarrow \Sigma G_1,\]
where $G_0,G_1$ belong to $\mathrm{Add}(\Sigma^p G|p\in\mathbb{Z})$, whose image under $F$ is a short exact sequence in $\mathrm{Grmod}(A)$
\[0\rightarrow FG_1\rightarrow FG_0\rightarrow FX\rightarrow 0.\]
If we apply $\mathrm{Hom}_{\mathcal{T}}(?,Y)$ to the triangle and $\mathrm{Hom}_{A}(?,FY)$ to the short exact sequence, we obtain a commutative
diagram with exact rows
$$\xymatrix@R0.6cm{ {\mathcal{T}}(\Sigma G_1,Y)\ar[r]&
{\mathcal{T}}(X,Y)\ar[r]\ar[d]& {\mathcal{T}}(G_0,Y)\ar[d]\ar[r]&
{\mathcal{T}}(G_1,Y)\ar[d]\\
0\ar[r]& (FX,FY)\ar[r]& (FG_0,FY)\ar[r]& (FG_1,FY) ,
}
$$
where the parentheses $(,)$ in the second row denote the groups of
$A$-linear maps. By Step~1, the rightmost two vertical maps are
isomorphisms. Therefore, the leftmost vertical map is surjective.
Since $X$ and $Y$ are arbitrary, we have proved that $F$ is full.
\end{proof}

\begin{Lem}\label{L:kernelsquare0}
 Let $J=\{f\in\mathrm{Mor}(\mathcal{T})|Ff=0\}$. Then $J^2=0$.
\end{Lem}
\begin{proof}
 Let $f:X\rightarrow Y$ be a morphism in $J$, that is, for any $p\in \mathbb{Z}$ and for any morphism $u: G\ra \Sigma^p X$, we have $\Sigma^p f\circ u=0$.

Let $G_1\stackrel{u}{\rightarrow} G_0\stackrel{v}{\rightarrow}
X\stackrel{w}{\rightarrow}\Sigma G_1$ be a triangle in $\mathcal{T}$
such that $G_0,G_1$ belong to
$\mathrm{Add}(\Sigma^{p}G|p\in\mathbb{Z})$. Since $f$ belongs to
$J$, we have $f\circ v=0$. Therefore, the morphism $f$ factors
through $w$, that is, there is
$f'\in\mathrm{Hom}_{\mathcal{T}}(\Sigma G_1,Y)$ such that $f=f'\circ
w$.

Let $G'_1\stackrel{u'}{\rightarrow} G'_0\stackrel{v'}{\rightarrow}
Y\stackrel{w'}{\rightarrow}\Sigma G'_1$ be a triangle in
$\mathcal{T}$ such that $G'_0,G'_1$ belong to
$\mathrm{Add}(\Sigma^{p}G|p\in\mathbb{Z})$, $Fu'$ is injective and $Fv'$ is surjective.
Then the induced homomorphism
$$
\mathrm{Hom}_{\mathcal{T}}(\Sigma G_1,G_0')\rightarrow
\mathrm{Hom}_{\mathcal{T}}(\Sigma G_1,Y)
$$
is surjective. Therefore, there is $h\in\mathrm{Hom}_{\mathcal{T}}(\Sigma G_1,G_0')$ such that
$f'=v'\circ h$.

Now let $g:Y\rightarrow Z$ be another morphism in $J$. By the arguments in the second paragraph there is $g':\Sigma G'_1\rightarrow Z$ such that $g=g'\circ w'$. Thus we have
 $g\circ f=g'\circ w'\circ v'\circ h\circ w =0$, and we are done.
\end{proof}

\subsection{Classically generated case} Let $k$ be a commutative ring and
let $\mathcal{T}$ be a $k$-linear triangulated category
with suspension functor $\Sigma$. Let $G$ be a {\em classical generator}
for $\mathcal{T}$, i.e. $\mathcal{T}$ is the closure of $G$ under taking
shifts, extensions and direct summands. Let
\[A=\bigoplus_{p\in\mathbb{Z}}\mathrm{Hom}_{\mathcal{T}}(G,\Sigma^{p}G)\]
 be the graded endomorphism algebra of $G$.
We assume that the category $\mathrm{grmod}(A)$ of finitely presented
graded $A$-modules is abelian (i.e. $A$ is graded right coherent) and hereditary.

\begin{Thm}\label{T:classificationclassical}
 The functor
 \[
 F:\mathcal{T}\rightarrow \mathrm{grmod}(A), X\mapsto \bigoplus_{p\in\mathbb{Z}}\mathrm{Hom}_{\mathcal{T}}(G,\Sigma^{p}X)
 \]
 is well-defined, full, essentially surjective, and its kernel has square zero.
 In particular, it induces a bijection from the set of isoclasses of objects
 (respectively, of indecomposable objects) of $\mathcal{T}$ to that of $\mathrm{grmod}(A)$.
\end{Thm}
\begin{proof}
Lemma~\ref{L:equivadd},~\ref{L:surj},~\ref{L:full} and
\ref{L:kernelsquare0} and their proofs are still valid, mutatis
mutandis. For example, we need to replace $\mathrm{Add}$ by
$\mathrm{add}$ in the statement of Lemma~\ref{L:equivadd}, where for
a class $\mathcal{S}$ of objects of an additive category
$\mathcal{A}$, we denote by $\mathrm{add}(\mathcal{S})$ the closure
of $\mathcal{S}$ under taking direct summands and finite direct
sums. It remains to prove that $F$ is well-defined, that is, for any
object $X$ of $\mathcal{T}$, the graded $A$-module $FX$ is indeed
finitely presented.

Let $\mathcal{T}'$ be the full additive subcategory of $\mathcal{T}$
consisting of  those objects $X$ such that $F(X)$ is a finitely
presented $A$-module. Evidently $G$ belongs to $\mathcal{T}'$. Thus,
in order to conclude that $\mathcal{T}'$ equals $\mathcal{T}$, it
suffices to show that $\mathcal{T}'$ is stable under shifts, direct
summands and extensions. The first two points are clear.

Suppose that we have a triangle
\[
Y\stackrel{u}{\rightarrow}Z\stackrel{v}{\rightarrow}X\stackrel{w}{\rightarrow}
\Sigma Y
\]
in $\mathcal{T}$ such that $FY$ and $FX$ are finitely presented.
Then the objects $$F(\Sigma^{-1}X)=(FX)[-1] \text{ and } F(\Sigma
Y)=(FY)[1]$$ are also finitely presented. We apply $F$ to the above
triangle to obtain an exact sequence
\[
\xymatrix{
F(\Sigma^{-1}X) \ar[rr]^{F(\Sigma^{-1}w)} & & F(Y)\ar[r]^{Fu} &
FZ\ar[r]^{Fv} &  FX \ar[r]^{Fw} & F(\Sigma Y). }
\]
Note that all components except possibly $FZ$ are finitely presented.
Since the category $\mathrm{grmod}(A)$ of finitely presented graded
$A$-modules is abelian, the kernel
$\mathrm{ker}Fv=\mathrm{coker}F(\Sigma^{-1}w)$ of $Fv$ and the image
$\mathrm{im}Fv=\mathrm{ker}Fw$ of $Fv$ are also finitely presented.
Consequently $FZ$ is finitely presented and $\mathcal{T}'$ is stable
under extensions. Therefore, the functor $F$ is well-defined.
\end{proof}

\begin{Examples}\label{exe}
a) Let $B$ be a finite dimensional hereditary algebra over a field
$k$. Let $\mathcal{T}=\mathcal{D}^{b}(\mathrm{mod}B)$ be the bounded
derived category of finite dimensional $B$-modules, and let $G$ be
the free $B$-module of rank $1$. Then $A=B$ and the functor
$F:\mathcal{D}^{b}(\mathrm{mod}B)\rightarrow \mathrm{grmod}(B)$
takes $X$ to its total homology $H^{*}X$.

b) Let $R$ be a discrete valuation ring with a uniformizing
parameter $\pi$. Denote $B=R/(\pi^2)$ and $k=R/(\pi)$. Let
$\mathcal{T}=\mathcal{D}^{b}(\mathrm{mod}B)$ be the bounded derived
category of finitely generated $B$-modules, and let $G$ be the
simple module $k$. Then the graded endomorphism algebra $A$ of $G$
in $\ct$ is isomorphic to the graded algebra $k[u]$ with
$\mathrm{deg}(u)=1$. Now Theorem~\ref{T:classificationclassical}
gives a new proof of K{\"u}nzer's classification \cite[Lemma
3.1]{Kuenzer} of the indecomposables of $\ct$ . In fact, using the
notations of K{\"u}nzer, up to isomorphism, the indecomposables are
the complexes $X^{[a,b]}$ and $X^{]-\infty, b]}$, where, for given
integers $a\leq b$, we denote by $X^{[a, b]}$ the complex
$$
\cdots \rightarrow 0\rightarrow \underbrace{B}_{a}
\stackrel{\pi}{\rightarrow} B \stackrel{\pi}{\rightarrow} \cdots
\stackrel{\pi}{\rightarrow} B \stackrel{\pi}{\rightarrow}
\underbrace{B}_b \rightarrow 0\rightarrow \cdots
$$
and by $X^{]-\infty,  b]}$ the complex
$$
\cdots \stackrel{\pi}{\rightarrow} B \stackrel{\pi}{\rightarrow}
B\stackrel{\pi}{\rightarrow} \underbrace{B}_b \rightarrow
0\rightarrow \cdots.
$$

c) Let $\tilde{A}$ be a differential graded  algebra (dg algebra for short) such that the category of finitely presented graded modules over the graded algebra $A=H^*(\tilde{A})$ is abelian and hereditary. Let $\mathcal{T}=\mathrm{per}(\tilde{A})$ be the perfect derived category and let $G$ be the free dg $\tilde{A}$-module of rank $1$. Then the functor $F$ takes $X$ to its total homology viewed as a graded $A$-module.
\end{Examples}

\section{Application of the classification}

Let $k$ be a field. Let $\Gamma$ denote the Ginzburg dg algebra of
type $A_1$ over $k$, i.e. $\Gamma$ is the dg algebra $k[t]$ with
$\mathrm{deg}(t)=-2$ and trivial differential.

Denote by $\mathrm{per}(\Gamma)$ the perfect derived category, i.e. the smallest thick subcategory
of the derived category $\mathcal{D}(\Gamma)$ containing $\Gamma$, and by
  $\mathcal{D}_{fd}(\Gamma)$ the finite dimensional derived category, i.e. the full triangulated subcategory
consisting of the dg $\Gamma$-modules whose homology is of finite
total dimension (\cf \cite{Keller94}). The triangulated category
$\mathcal{D}_{fd}(\Gamma)$ is Hom-finite and $3$-Calabi-Yau (\cf
\cite{Joergensen04} or \cite{Keller08}), classically generated by
the simple dg $\Gamma$-module $S=\Gamma/(t\Gamma)$ concentrated in
degree $0$, which is a spherical object of dimension $3$.

Let $[1]$ denote the shift functor of the category $\mathrm{grmod}(\Gamma)$ of finitely presented graded $\Gamma$-modules. For an integer $p$ and a strictly positive integer $n$, the finite dimensional graded $\Gamma$-module $\Gamma/(t^n\Gamma)[p]$, viewed as an object in $\mathcal{D}_{fd}(\Gamma)$, is indecomposable.

\begin{Thm}[J\o rgensen \cite{Joergensen04}]\label{T:classification}
a) Each indecomposable object in $\mathrm{per}(\Gamma)$ is isomorphic to either
$\Gamma/(t^n\Gamma)[p]$ for some integer $p$ and some strictly positive integer $n$ or $\Gamma[p]$ for some integer $p$.

b) Each indecomposable object in $\mathcal{D}_{fd}(\Gamma)$ is isomorphic to
$\Gamma/(t^n\Gamma)[p]$ for some integer $p$ and some strictly positive integer $n$.

\end{Thm}

\begin{proof} It is readily seen that the category $\mathrm{grmod}(A)
$ for $A=H^*(\Gamma)(=\Gamma$ as graded algebras) is abelian and hereditary.
 We are therefore in a particular case of Example~\ref{exe} c). The functor $$F=H^*:\mathrm{per}(\Gamma)\rightarrow \mathrm{grmod}(\Gamma)$$  induces a bijection between the set of isoclasses of indecomposable objects of $\mathrm{per}(\Gamma)$ and that of $\mathrm{grmod}(A)$. Moreover, the full subcategory $\mathcal{D}_{fd}(\Gamma)$ of $\mathrm{per}(\Gamma)$ is sent by $F$ to the full subcategory   of $\mathrm{grmod}(\Gamma)$ consisting of finite dimensional graded $\Gamma$-modules. Now the theorem follows from the classification of indecomposable objects for the latter category, which is well-known.
\end{proof}

\begin{Rem}
It is not hard to check that the Auslander-Reiten quiver
of the perfect derived category has the following shape
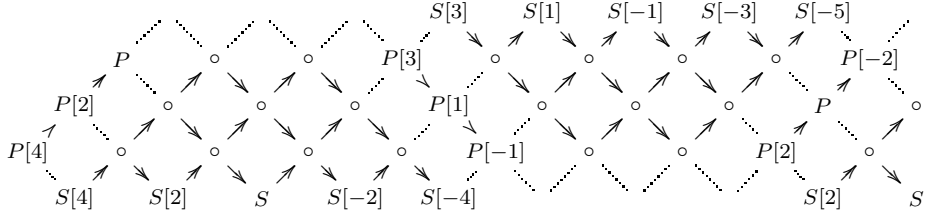
\[
{\scriptsize
\begin{xy} 0;<0.35pt,0pt>:<0pt,-0.35pt>::
(0,150) *+{P[4]} ="0",
(0,50) *+{} ="1",
(50,200) *+{S[4]} ="2",
(50,100) *+{P[2]} ="3",
(50,0) *+{} ="4",
(100,150) *+{\circ} ="5",
(100,50) *+{P} ="6",
(150,200) *+{S[2]} ="7",
(150,100) *+{\circ} ="8",
(150,0) *+{} ="9",
(200,150) *+{\circ} ="10",
(200,50) *+{\circ} ="11",
(250,200) *+{S} ="12",
(250,100) *+{\circ} ="13",
(250,0) *+{} ="14",
(300,150) *+{\circ} ="15",
(300,50) *+{\circ} ="16",
(350,200) *+{S[-2]} ="17",
(350,100) *+{\circ} ="18",
(350,0) *+{} ="19",
(400,150) *+{\circ} ="20",
(400,50) *+{P[3]} ="21",
(450,200) *+{S[-4]}="22",
(450,100) *+{P[1]} ="23",
(450,0) *+{S[3]} ="24",
(500,150) *+{P[-1]} ="25",
(500,50) *+{\circ} ="26",
(550,200) *+{} ="27",
(550,100) *+{\circ} ="28",
(550,0) *+{S[1]} ="29",
(600,150) *+{\circ} ="30",
(600,50) *+{\circ} ="31",
(650,200) *+{} ="32",
(650,100) *+{\circ} ="33",
(650,0) *+{S[-1]} ="34",
(700,150) *+{\circ} ="35",
(700,50) *+{\circ} ="36",
(750,200) *+{} ="37",
(750,100) *+{\circ} ="38",
(750,0) *+{S[-3]} ="39",
(800,150) *+{P[2]} ="40",
(800,50) *+{\circ} ="41",
(850,200) *+{S[2]} ="42",
(850,100) *+{P} ="43",
(850,0) *+{S[-5]} ="44",
(900,150) *+{\circ} ="45",
(900,50) *+{P[-2]} ="46",
(950,200) *+{S} ="47",
(950,100) *+{\circ} ="48",
(950,0) *+{} ="49",
"0", {\ar@{.}"2"},
"0", {\ar"3"},
"1", {\ar@{}"3"},
"1", {\ar@{}"4"},
"2", {\ar"5"},
"3", {\ar@{.}"5"},
"3", {\ar"6"},
"4", {\ar@{}"6"},
"5", {\ar"7"},
"5", {\ar"8"},
"6", {\ar@{.}"8"},
"6", {\ar@{.}"9"},
"7", {\ar"10"},
"8", {\ar"10"},
"8", {\ar"11"},
"9", {\ar@{.}"11"},
"10", {\ar"12"},
"10", {\ar"13"},
"11", {\ar"13"},
"11", {\ar@{.}"14"},
"12", {\ar"15"},
"13", {\ar"15"},
"13", {\ar"16"},
"14", {\ar@{.}"16"},
"15", {\ar"17"},
"15", {\ar"18"},
"16", {\ar"18"},
"16", {\ar@{.}"19"},
"17", {\ar"20"},
"18", {\ar"20"},
"18", {\ar@{.}"21"},
"19", {\ar@{.}"21"},
"20", {\ar"22"},
"20", {\ar@{.}"23"},
"21", {\ar"23"},
"21", {\ar@{.}"24"},
"22", {\ar@{.}"25"},
"23", {\ar"25"},
"23", {\ar@{.}"26"},
"24", {\ar"26"},
"25", {\ar@{.}"27"},
"25", {\ar@{.}"28"},
"26", {\ar"28"},
"26", {\ar"29"},
"27", {\ar@{.}"30"},
"28", {\ar"30"},
"28", {\ar"31"},
"29", {\ar"31"},
"30", {\ar@{.}"32"},
"30", {\ar"33"},
"31", {\ar"33"},
"31", {\ar"34"},
"32", {\ar@{.}"35"},
"33", {\ar"35"},
"33", {\ar"36"},
"34", {\ar"36"},
"35", {\ar@{.}"37"},
"35", {\ar"38"},
"36", {\ar"38"},
"36", {\ar"39"},
"37", {\ar@{.}"40"},
"38", {\ar@{.}"40"},
"38", {\ar"41"},
"39", {\ar"41"},
"40", {\ar@{.}"42"},
"40", {\ar"43"},
"41", {\ar@{.}"43"},
"41", {\ar"44"},
"42", {\ar"45"},
"43", {\ar@{.}"45"},
"43", {\ar"46"},
"44", {\ar@{.}"46"},
"45", {\ar"47"},
"45", {\ar"48"},
"46", {\ar@{.}"48"},
"46", {\ar@{.}"49"},
\end{xy}
}
\]
where the picture is periodic as indicated by the labels. The
Auslander-Reiten quiver of $\mathcal{D}_{fd}(\Gamma)$ is the
subquiver consisting of the components containing the simples $S$
and $S[1]$. This latter quiver was first determined by
P.~J\o rgensen in \cite{Joergensen04}; he considerably generalized the
result in \cite{Joergensen06}.
\end{Rem}

\section{The Hall algebra}
In this section, we prove the structure theorem
(Theorem~\ref{T:structure}) for the (derived) Hall algebra of the the
Ginzburg dg algebra of type $A_1$.  We begin with some reminders on
Hall algebras of triangulated categories. We refer to
\cite{Schiffmann06} for an excellent introduction to
non derived Hall algebras.

\bigskip
\subsection{The Hall algebra}
We follow \cite{Toen06} and \cite{XiaoXu06}. Let $\mathbb{Q}$ be the field of rational numbers, $q$ be a prime power and $\mathbb{F}_q$ be the finite field with $q$ elements.
Let $\mathcal{C}$ be a Hom-finite triangulated $\mathbb{F}_q$-category with suspension functor $\Sigma$, such that for all objects $X$ and $Y$ of $\mathcal{C}$, the space of morphisms from $X$ to $\Sigma^{-i}Y$ vanishes for all but finitely many positive integers $i$.

Let $X$, $Y$ and $Z$ be three objects of $\mathcal{C}$. We denote by $\mathrm{Aut}(Y)$ the group of automorphisms of $Y$ and by $[Y,Z]_X$ the set of morphisms from $Y$ to $Z$ with cone isomorphic to $X$.
Following~\cite{Toen06}, we define the Hall number by
\[F_{XY}^Z=\frac{|[Y,Z]_X|}{|\mathrm{Aut}(Y)|}\cdot\frac{\prod_{i>0}|\mathrm{Hom}(Y,\Sigma^{-i}Z)|^{(-1)^i}}{\prod_{i>0}|\mathrm{Hom}(Y,\Sigma^{-i}Y)|^{(-1)^i}},\]
where $|\cdot|$ denotes the cardinality. The {\em Hall algebra} of $\mathcal{C}$ over $\mathbb{Q}$, denoted by $\mathcal{H}(\mathcal{C})$, is the $\mathbb{Q}$-vector space with basis the isoclasses $[X]$ of objects $X$ of $\mathcal{C}$ whose multiplication is given by
\[[X][Y]=\sum_{[Z]}F_{XY}^Z [Z].\] It is shown
in~\cite{Toen06}~\cite{XiaoXu06} that it is an associative algebra
with unit $[0]$. Notice however that the algebra we define here is
opposite to that in~\cite{Toen06}~\cite{XiaoXu06}.

\subsection{The structure theorem}

Let $A$ be the $\mathbb{Q}$-algebra with generators $x_i$ and $y_i$, $i\in\mathbb{Z}$, subject to the following relations:
\begin{eqnarray}
&x_i^2 x_{i-1} - (1+q^{-1})x_i x_{i-1} x_i + q^{-1}x_{i-1}x_i^2&\label{x1}\\
&x_i x_{i-1}^2 - (1+q^{-1})x_{i-1} x_i x_{i-1} + q^{-1} x_{i-1}^2 x_i&\label{x2}\\
&x_i x_j - x_j x_i \qquad\text{ if } |i-j|>1&\label{x3}\\
&y_i x_i - q x_i y_i + \frac{q}{q-1}&\label{ore1}\\
&y_i x_{i+1} - q^{-1} x_{i+1} y_i - \frac{1}{q-1}&\label{ore2}\\
&y_i x_j - x_j y_i  \qquad\text{ if } j\neq i,i+1&\label{ore3}\\
&y_i^2 y_{i-1} - (1+q^{-1})y_i y_{i-1} y_i + q^{-1}y_{i-1}y_i^2&\label{y1}\\
&y_i y_{i-1}^2 - (1+q^{-1})y_{i-1} y_i y_{i-1} + q^{-1} y_{i-1}^2 y_i&\label{y2}\\
&y_i y_j - y_j y_i \qquad\text{ if } |i-j|>1.\label{y3}&
\end{eqnarray}

Let $\Gamma$ be the Ginzburg dg algebra of type $A_1$ over the finite
field $\mathbb{F}_q$, and $\mathcal{D}_{fd}(\Gamma)$ the finite
dimensional derived category with suspension
functor $\Sigma$. Let
$\mathcal{H}=\mathcal{H}(\mathcal{D}_{fd}(\Gamma))$ be the
Hall algebra.

\begin{Thm}\label{T:structure} We have a $\mathbb{Q}$-algebra isomorphism
\[
\phi: A\longrightarrow \mathcal{H},\qquad
x_i\mapsto [\Sigma^{-2i}S], y_i\mapsto [\Sigma^{-2i-1}S],
\]
where we recall that $S=\Gamma/(t\Gamma)$ is the simple dg $\Gamma$-module
concentrated in degree $0$.
\end{Thm}

One checks by a direct computation that $\phi$ is indeed an algebra
homomorphism, i.e. the relations (\ref{x1})--(\ref{y3}) are satisfied
if we replace $x_i$ and $y_i$ by $[\Sigma^{-2i}S]$ and
$[\Sigma^{-2i-1}S]$ respectively. It remains to prove the surjectivity
and the injectivity.


\bigskip
\subsection{Surjectivity of $\phi$}

\begin{Prop}
 The $\mathbb{Q}$-algebra $\mathcal{H}$ is generated by the $[\Sigma^{p}S],p\in\mathbb{Z}$.
\end{Prop}
\begin{proof} 




Let $M$ be an object of $\mathcal{D}_{fd}(\Gamma)$.
Suppose $M\cong \bigoplus_{p\leq p_0}M_{p}$, where $p_0$ is an integer, $M_p$ is a dg $\Gamma$-module which has trivial differential and which is generated in degree $p$, and $M_{p_0}$ is nontrivial. Without loss of generality, we may assume $p_0=0$. Write $L=M_{0}$ and $N=\bigoplus_{p<0}M_{p}$.

We have a triangle
\[\tau_{\leq -1}(L) \oplus N\rightarrow M \rightarrow \tau_{\geq 0}(L) \rightarrow \Sigma (\tau_{\leq -1}(L) \oplus N),\]
which gives rise to a short exact sequence
\[0\rightarrow \mathrm{rad}(H^*L) \oplus H^*N\rightarrow H^*M \rightarrow \mathrm{top}(H^*L) \rightarrow 0.\]

Now let $E$ be an extension of $\tau_{\geq 0}(L)$ by $\tau_{\leq -1}(L) \oplus N$, i.e. we have a triangle
\[\tau_{\leq -1}(L) \oplus N\rightarrow E \rightarrow \tau_{\geq 0}(L) \stackrel{f}{\rightarrow} \Sigma (\tau_{\leq -1}(L) \oplus N),\]
which gives rise to a long exact sequence
\[
\xymatrix@R=0.2cm{\mathrm{top}(H^*L)[-1] \ar[rr]^-{(H^*f)[-1]} &  &\mathrm{rad}(H^*L) \oplus H^*N \ar[r] & H^*E\\
& \mbox{} \ar[r] &\mathrm{top}(H^*L) \ar[r]^-{H^*f} & \mathrm{rad}(H^*L)[1] \oplus (H^*N)[1].
}
\]
If $H^*f\neq 0$, then the dimensions of $H^* E$ and $H^* M$ are equal.
If $H^*f=0$, then the dimensions are equal but $H^*E$ has more
indecomposable direct summands than $H^*M$, and hence by
Theorem~\ref{T:classificationclassical}, the object $E$ has more
indecomposable direct summands than $M$; moreover, the object $E$
has precisely the same number of indecomposable direct summands as
$M$ if and only if $E\cong M$, since this holds if we replace $E$
and $M$ by $H^*E$ and $H^*M$ respectively. By induction the proof
reduces to the case where $H^*M$ is concentrated in one degree,
namely, the case where $M$ is isomorphic to the direct sum of $m$
copies of $S$ for some positive integer $m$. But we have
\[
[S^{\oplus m}]=\frac{q-1}{q^m-1}[S]^m,
\]
which finishes the proof.
\end{proof}

As a consequence,  we have
\begin{Cor}
 The algebra homomorphism $\phi:A\rightarrow \mathcal{H}$ is surjective.
\end{Cor}

\bigskip
\subsection{Injectivity of $\phi$}
Let $R_x$ (respectively, $R_y$) be the subalgebra of $A$ generated by
$\{x_i|i\in\mathbb{Z}\}$ (respectively, by $\{y_i|i\in\mathbb{Z}\}$).
The image of $R_x$ under $\phi$ is the subalgebra of $\mathcal{H}$
generated by $\{[\Sigma^{2i}S]|i\in\mathbb{Z}\}$, denoted by
$\mathcal{H}_x$, which has a $\mathbb{Q}$-basis
$\{[M]|M\in\mathcal{D}_{fd}(\Gamma),H^{odd}M=0\}$, where $H^{odd}M$ is
the direct sum of homology spaces of $M$ in odd degrees. Similarly,
the image $\mathcal{H}_y$ of $R_y$ under $\phi$ is the subalgebra of
$\mathcal{H}$ generated by $\{[\Sigma^{2i+1}S]|i\in\mathbb{Z}\}$, and
has a $\mathbb{Q}$-basis
$\{[M]|M\in\mathcal{D}_{fd}(\Gamma),H^{even}M=0\}$, where $H^{even}M$
is the direct sum of homology spaces of $M$ in even degrees.

Thanks to (\ref{ore1})(\ref{ore2})(\ref{ore3}), we have an isomorphism
of $\mathbb{Q}$-vector spaces
\[
\psi:R_x\otimes R_y\rightarrow A,\qquad f(x)\otimes g(y)\mapsto f(x)g(y).
\]
In particular, the product of a basis of $R_x$ and a basis of $R_y$ is
a basis of $A$.  Now the injectivity is implied by the following two
lemmas.

\begin{Lem}\label{L:isomhalves}
The restriction $\phi|_{R_x}:R_x\rightarrow \mathcal{H}_x$
(respectively, $\phi|_{R_y}:R_y\rightarrow \mathcal{H}_y$) is an
isomorphism.
\end{Lem}

\begin{Lem}\label{L:basis} The set $\{[M][N]| M,N\in\mathcal{D}_{fd}(\Gamma),
H^{odd}M=H^{even}N=0\}$ is a $\mathbb{Q}$-basis of
$\mathcal{H}$.
\end{Lem}

\begin{proof}[Proof of Lemma~\ref{L:isomhalves}:]

Let $\cp$ be the path category of the quiver $\vec{A}_\infty$
whose vertices are the integers and which has one arrow
from $n$ to $n+1$ for each integer $n$. We have a fully
faithful functor $F$ from $\cp$ to $\per\Gamma$ taking
the object $n$ to the dg module $P[-2n]$. Let $F^*$
be the functor $\cd\Gamma\to \Mod\cp$ taking an object
$X$ to the module $n \mapsto \Hom(Fn, X)$. It induces
an equivalence from the full subcategory whose objects
are the dg modules $M$ whose homology is finite-dimensional
and vanishes in odd degrees to the category $\mod \cp$ of
finite-dimensional representations of the quiver $\vec{A}_\infty$.
It is clear that this functor induces a bijection between
from the basis of $\ch_x$ to that of the Hall algebra of $\mod \cp$,
i.e. the Hall algebra $\ch(\vec{A}_\infty)$ of the quiver $\vec{A}_\infty$.
The claim now follows since it is well-known that the composition
\[
R_x \to \ch_x \to \ch(\vec{A}_\infty)
\]
is an algebra isomorphism.
\end{proof}

\begin{proof}[Proof of Lemma~\ref{L:basis}:]

By the surjectivity of $\phi$, the set of products
\[
\{[M][N]| M,N\in\mathcal{D}_{fd}(\Gamma), H^{odd}M=H^{even}N=0\}
\]
generates the $\mathbb{Q}$-vector space $\mathcal{H}$.
It remains to prove that these products are linearly independent.

Following~\cite{JensenSuZimmermann05}, we define a partial order
$\leq_{\Delta}$ on the set of isoclasses of objects in
$\mathcal{D}_{fd}(\Gamma)$ as follows: if $X$ and $Y$ are two objects
of $\mathcal{D}_{fd}(\Gamma)$, then $[Y]\leq_{\Delta} [X]$ if there
exists an object $Z$ of $\mathcal{D}_{fd}(\Gamma)$ and a triangle in
$\mathcal{D}_{fd}(\Gamma)$:
\[X\rightarrow Y\oplus Z\rightarrow Z\rightarrow \Sigma X.\]
We extend the partial order $\leq_{\Delta}$ to a total order $\preceq$.

Now suppose $(M_1,N_1),\ldots,(M_r,N_r)$ are pairwise distinct pairs
of objects of $\mathcal{D}_{fd}(\Gamma)$ such that
$$H^{odd}M_1=\ldots=H^{odd}M_r=H^{even}N_1=\ldots=H^{even}N_r=0.$$
Suppose that
$\lambda_1,\ldots,\lambda_r$ are rational numbers such that
\[
\lambda_1 [M_1][N_1]+\ldots+\lambda_r [M_r][N_r]=0.
\]
By the assumption on the $M_i$'s and $N_i$'s, there is a unique
maximal element among all $[M_i\oplus N_i]$'s, say $[M_1\oplus N_1]$.
Then we have
\[
\lambda_1 [M_1][N_1]+\ldots+\lambda_r [M_r][N_r]=\lambda_1
F_{M_1N_1}^{M_1\oplus N_1} [M_1\oplus N_1] + \text{smaller terms},
\]
since a nontrivial extension of two objects is always smaller than the
direct sum of them.  The (derived) Hall number $F_{M_1N_1}^{M_1\oplus
N_1}$ is a nonzero rational number. Therefore $\lambda_1$ has to be
zero. An induction on $r$ shows that $\lambda_1=\ldots=\lambda_r=0$.
\end{proof}

\section{From the Hall algebra to the torus}\label{S:cluster-variety}

Let $v=\sqrt{q}$. We tensor $A$ with $\mathbb{Q}(v)$ over
$\mathbb{Q}$, and still denote the resulting algebra by $A$.  Let
$I$ be the ideal of $A$ generated by the space $[A,A]$ of commutators
of $A$.

\begin{Lem}
  The assignment $\varphi: x_i\mapsto \frac{v}{v^2-1}x,
  y_i\mapsto\frac{v}{v^2-1}x^{-1}$ defines an algebra homomorphism
  from $A$ to $\mathbb{Q}(v)[x,x^{-1}]$ with kernel $I$.
\end{Lem}

\begin{proof} We have
\begin{eqnarray*} A/I &\cong&
\mathbb{Q}(v)[x_i,y_i]_{i\in\mathbb{Z}}/(x_i
y_i=x_{i+1}y_i=\frac{q}{(q-1)^2}).
\end{eqnarray*}
Now it is clear that $x_i\mapsto\frac{v}{v^2-1}x,
y_i\mapsto\frac{v}{v^2-1}x^{-1}$ defines an algebra isomorphism from
$A/I$ to $\mathbb{Q}(v)[x,x^{-1}]$.
\end{proof}

\section{General case}
\label{S:arbitrary-dimension}

The general case can be treated similarly.
Here we only give the final result and the key point of the proof.

\begin{Thm} Let $d$ be an integer, and $d'=d-1$. Let $\mathcal{D}$ be the algebraic
  triangulated category classically generated by a $d$-spherical
  object, and let $\mathcal{H}(\mathcal{D})$ be the Hall algebra of
  $\mathcal{D}$ over $\mathbb{Q}$.
\bigskip

(i) When $d\geq 3$ (i.e. $d'\geq 2$), the algebra $\mathcal{H}(\mathcal{D})$ is generated by $z_i$,
  $i\in\mathbb{Z}$, subject to the following relations:
\begin{eqnarray}
&z_i^2 z_{i-d'} - (q+1)q^{(-1)^{d}}z_i z_{i-d'} z_i +
q^{1+2(-1)^{d}}z_{i-d'}z_i^2&\\
&z_i z_{i-d'}^2 - (q+1)q^{(-1)^{d}}z_{i-d'} z_i z_{i-d'} + q^{1+2(-1)^{d}} z_{i-d'}^2 z_i&\\
&z_i z_{i+1} - q^{-1} z_{i+1} z_i - \frac{1}{q-1}&\\
&z_i z_j - q^{(-1)^{j-i}(1+(-1)^{-d})}z_j z_i \qquad\text{ if } i-j\leq -d &\\
&z_i z_j - q^{(-1)^{j-i}} z_j z_i \qquad\text{ if } -d'<i-j<-1.&
\end{eqnarray}
\bigskip

(ii) When $d=2$ (i.e. $d'=1$), the algebra $\mathcal{H}(\mathcal{D})$ is generated by $z_i$,
  $i\in\mathbb{Z}$, subject to the following relations:
\begin{eqnarray}
&z_i^2 z_{i-1} - (q+1)qz_i z_{i-1} z_i +
q^{3}z_{i-1}z_i^2-q(q+1)z_i&\label{z1}\\
&z_i z_{i-1}^2 - (q+1)qz_{i-1} z_i z_{i-1} + q^{3} z_{i-1}^2 z_i-q(q+1)z_{i-1}&\\
&z_i z_j - q^{2(-1)^{j-i}}z_j z_i \qquad\text{ if } i-j\leq -2. &
\end{eqnarray}
\bigskip

(iii) When $d=1$ (i.e. $d'=0$), the algebra $\mathcal{H}(\mathcal{D})$ is generated by $z_{i,j}$,
  $i\in\mathbb{Z}$, $j\in\mathbb{N}$, subject to the following relations:
\begin{eqnarray}
  &z_{i, j}z_{i', j'}-z_{i', j'}z_{i, j},& \text{if}\ i-i'\neq \pm 1\\
 &z_{i, j}z_{i+1, j'}-\sum_{0\leq l\leq
 \mathrm{min}\{j, j'\}}F_{j, j'}^{l}z_{i+1, j'-l}z_{i, j-l}&
\end{eqnarray}
where
\[F_{j, j'}^{l}=\begin{cases} 1, & \text{ if } l=0\\
                           \frac{q-1}{q^{l+1}}, & \text{ if } 0<l<\mathrm{min}\{j,j'\}\\
                           q^{-j'}, & \text{ if } l=j'<j\\
                           q^{-j}, & \text{ if } l=j<j'\\
                           \frac{1}{q^{j-1}(q-1)}, & \text{ if } l=j=j'.
             \end{cases}
\]

\bigskip

(iv) When $d=0$ (i.e. $d'=-1$), the algebra $\mathcal{H}(\mathcal{D})$ is generated by $z_i$, $z'_i$, $i\in\mathbb{Z}$, subject to commutative relations, that is, $\mathcal{H}(\mathcal{D})$ is isomorphic to the polynomial algebra $\mathbb{Q}[z_i,z'_i|i\in\mathbb{Z}]$.
\bigskip

(v) When $d\leq -1$ (i.e. $d'\leq -2$),  the algebra $\mathcal{H}(\mathcal{D})$ is generated by $z_i$,
  $i\in\mathbb{Z}$, subject to the following relations:
\begin{eqnarray}
&z_i^2 z_{i-d'} - (q+1)q^{-1-(-1)^{-d}}z_i z_{i-d'} z_i +
q^{-1-2(-1)^{-d}}z_{i-d'}z_i^2&\\
&z_i z_{i-d'}^2 - (q+1)q^{-1-(-1)^{-d}}z_{i-d'} z_i z_{i-d'} + q^{-1-2(-1)^{-d}} z_{i-d'}^2 z_i&\\
&z_i z_{i+1} - q^{-1} z_{i+1} z_i - \frac{1}{q^{(-1)^{-d}}(q-1)}&\\
&z_i z_j - q^{(-1)^{j-i}(1+(-1)^{-d})}z_j z_i \qquad\text{ if } i-j< d'&\\
&z_i z_j - q^{(-1)^{j-i}} z_j z_i \qquad\text{ if } d\leq i-j<-1.&
\end{eqnarray}
\end{Thm}
\begin{proof} Let $S$ be the $d$-spherical object, and $\Sigma$ be the suspension functor.

(i) and (v): similar to Theorem~\ref{T:structure}, with $z_i$ representing $\Sigma^{-i} S$.

(ii) Notice that both the Hall algebra $\mathcal{H}(\mathcal{D})$ and the desired algebra are filtered, and the algebra homomorphism from the desired algebra to the Hall algebra $\mathcal{H}(\mathcal{D})$ is a morphism of filtered algebras, and the associated graded algebra homomorphism is an isomorphism, which has a similar proof to that for Theorem~\ref{T:structure}, with $z_i$ representing for $\Sigma^{-i} S$.

(iii) In this case, the triangulated category $\mathcal{D}$ is
equivalent to the bounded derived category of the hereditary abelian
category of finite dimensional re\-pre\-sentations over the Jordan quiver.
Then the desired result follows from~\cite[Proposition 7.1]{Toen06}
and the classical result on the Hall algebra of the above hereditary
abelian category (\cf for example~\cite{Macdonald95}), with $z_{i,j}$ representing $\Sigma^{-i} M_j$, where $M_j$ is the indecomposable nilpotent representation of the Jordan quiver of dimension $j$.

(iv) In this case, the triangulated category $\mathcal{D}$ is equivalent to the bounded derived category of the semisimple abelian category with two simple objects $T$ and $T'$. In the statement, $z_i$ represents $\Sigma^{-i} T$, and $z'_i$ represents $\Sigma^{-i} T'$.
\end{proof}



\def\cprime{$'$}
\providecommand{\bysame}{\leavevmode\hbox to3em{\hrulefill}\thinspace}
\providecommand{\MR}{\relax\ifhmode\unskip\space\fi MR }
\providecommand{\MRhref}[2]{%
  \href{http://www.ams.org/mathscinet-getitem?mr=#1}{#2}
}
\providecommand{\href}[2]{#2}

\end{document}